\documentclass[usenames,dvipsnames,a4paper,reqno]{amsart}
\usepackage[utf8]{inputenc}
\usepackage{preamble}
\usepackage{semtkzX}

\title{Models of Bihyperelliptic Curves}
\author{Omri Faraggi}
\address{Department of Mathematics, University College London, London WC1H 0AY, UK}
\email{omri.faraggi.17@ucl.ac.uk}
\date{February 2021}

\begin{document}
	
\maketitle

\begin{abstract}
	We give an explicit description of the minimal regular model of bihyperelliptic curves with semistable reduction over a local field of odd residue characteristic. We do this using a generalisation of the cluster picture; a completely combinatorial object attached to a hyperelliptic curve $y^2 = f(x)$ over $K$ which contains the data of the $p$-adic distances between the roots of $f$. We add some information, resulting in a chromatic cluster picture, and show that this determines the minimal regular model of $Y$ with the action of Frobenius.
\end{abstract}

\renewcommand{\contentsname}{Table of contents}
\thispagestyle{empty}\setcounter{tocdepth}{1}
\tableofcontents

\section{Introduction}

\subsection{Summary} In this paper we give an explicit description of the minimal regular model of a bihyperelliptic curve $Y$ with semistable reduction over a local field $K$ of residue characteristic $p>2$. A \textit{bihyperelliptic curve} is a smooth curve with maps to two distinct hyperelliptic curves\footnote{For us, a hyperelliptic curve of genus $g$ is a curve given by equations $y^2 = f(x)$ and $v^2 = u^{2g+2} f(1/u)$ glued along the maps $(x, y) \mapsto (1/x, y/x^{g+1})$.} $C_1$ and $C_2$. This is achieved using a generalisation of the\textit{ cluster picture}, an invariant of hyperelliptic curves.  Cluster pictures are a relatively recent development; they are a completely combinatorial object attached to a hyperelliptic curve $C/K$ from which much of the local arithmetic of $C$ can be deduced. For example, cluster pictures have been used to calculate the Galois representation, semistable model, conductor and minimal discriminant of $C$ in \cite{DDMM18}, the minimal snc model if $C$ has tame reduction in \cite{faraggi20models}, the Tamagawa number in \cite{betts2018computation}, the root number in \cite{bisatt2019clusters} and finally differentials in \cite{kunzweiler2019differential} and \cite{muselli2020models}. A survey article is available at \cite{best2020users}. Our generalisation, the \textit{chromatic cluster picture} of $Y$, is an amalgamation of the cluster pictures of $C_1$ and $C_2$, and is mostly sufficient to determine the model of $Y$.

\subsection{Bihyperelliptic Curves, Models and Clusters} Our main objects of study will be bihyperelliptic curves, a curve $Y$ with maps to two distinct hyperelliptic curves $C_1$ and $C_2$. 

\begin{definition}
	\label{def:bihyp}
	Let $C_1 : y_1^2 = f_1(x)$ and $C_2 : y^2_2 = f_2(x)$ be hyperelliptic curves, given by their affine models, with $f_1$ and $f_2$ coprime. Then $C_h : y_h^2 = f_1(x)f_2(x)$ is their \textit{composite curve}, also hyperelliptic, and the curve $Y$, the smooth projective closure of $$ \left\{ \begin{array}{c}
	y_1^2 = f_1(x) \\ y^2_2 = f_2(x)
	\end{array}\right\},$$ is a \textit{bihyperelliptic curve}. The curves fit into a tower:

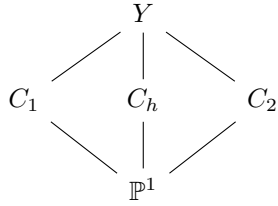
\begin{figure}[h]
	\centering
		\begin{tikzcd}
			& Y \arrow[d, no head] \arrow[rd, no head] &                              \\
			C_1 \arrow[ru, no head] \arrow[rd, no head] & C_h \arrow[d, no head]     & C_2 \\
			& \PP^1 \arrow[ru, no head]                             &                             
		\end{tikzcd}

\end{figure}

such that $Y/\PP^1$ is a Galois cover with Galois group $C_2 \times C_2$.
\end{definition}

%
%
%

We would like to find the minimal regular model $\Y^{\textrm{min}}$ of $Y$ in the case where $Y$ has semistable reduction. Recently there has been a flurry of activity finding explicit descriptions for models of different classes of curves, including \cite{DDMM18}, \cite{faraggi20models}, \cite{ruth2015models}, \cite{muselli2020models} (hyperelliptic curves), \cite{bouw2017computing} (superelliptic), \cite{lercier2018reduction}, \cite{bouw2020reduction} (genus 3), \cite{dokchitser2018models} ($\Delta_v$-regular).

To a hyperelliptic curve such as $C_1$ we attach a cluster picture, following \cite{DDMM18}.

\begin{definition}
	\label{def:cluster}
	Let $C : y^2 = f(x)$ be a hyperelliptic curve over $K$, with $\Rcal$ the set of roots of $f$. A \textit{cluster} is a non-empty subset $\s \subseteq \mathcal{R}$ of the form $\s = D \cap \Rcal$ for some disc $D = z + \pi^n \mathcal{O}_K$, where $z \in \overline{K}$ and $n \in \Q$. If $\s$ is a cluster and $|\s| > 1$, $\s$ is a \textit{proper} cluster and we define its \textit{depth} $$d_{\s} = \min_{r, r' \in \s} v_K(r-r').$$
	The \textit{cluster picture} $\Sigma = \Sigma_{C/K}$ is the set of all clusters of the roots of $f$.
\end{definition}

The chromatic cluster picture $\Sigma_{\chi}$ of $Y$ is the cluster picture $\Sigma_h$ of $C_h$, the composite curve of $C_1$ and $C_2$, along with a colouring function $c : \Rcal \rightarrow \{\textrm{red, blue}\}$, such that the roots coming from $C_1$ (resp. $C_2$) are coloured red (resp. blue). This information is sufficient to calculate the dual graph of the special fibre $\Y^{\textrm{min}}_k$ of the minimal regular model. 

\begin{thm}
	Let $K$ be a local field of odd residue characteristic and let $C_1 :y_1^2 = f_1(x)$ and $C_2: y_2^2 = f_2(x)$ be two hyperelliptic curves over $K$. Let $C_h$ be their composite curve. Let $Y$ be the bihyperelliptic curve arising from $C_1$ and $C_2$, such that $Y$ has semistable reduction and all the depths in the chromatic cluster picture of $Y$ are integers. Then the dual graph with genera of the special fibre $\Y^{\mathrm{min}}$ of the minimal regular model of $Y$ is entirely determined by the chromatic cluster picture of $Y$.
\end{thm}

\begin{eg}
	\label{eg:twinintrochrom}
	Let $Y$ be any bihyperelliptic curve with the chromatic cluster picture below, where the number to the bottom right of a cluster $\s$ indicates the relative depth $\delta_{\s}$ of the cluster: that is, $d_{\s} - d_{P(\s)}$ where $P(\s)$ is the cluster immediately containing $\s$. The dual graph of $\Y^{\textrm{min}}_k$ for any such $Y$ is shown below.
	
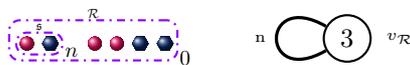
\begin{figure}[ht]
		\centering
		\clusterpicture
		\Root[A] {}{first}{r1};
		\Root[B] {}{r1}{r2};
		\ClusterLDName c1[][n][\s] = (r1)(r2), clB;
		\Root[A] {}{c1}{r3};
		\Root[A] {}{r3}{r4};
		\Root[B] {}{r4}{r5};
		\Root[B] {}{r5}{r6};
		\ClusterLDName c2[][0][\Rcal] = (c1)(r3)(r4)(r5)(r6), clB;
		\endclusterpicture \quad \quad
		\raisebox{-1.8em}{\begin{tikzpicture}
			\tikzstyle{every node}=[font=\tiny]
			\node at (0.7, 0) {$v_{\Rcal}$};
			\begin{scope}[every node/.style={circle,thick,draw}]
			\node (A) at (0,0) {3};
			\end{scope}
			
			\path (A) edge[very thick, out=150, in=210, looseness=8] node[left] {n} (A);
			
			\end{tikzpicture}}
	\end{figure}
	
\end{eg}

We prove this theorem by explicitly describing the dual graph. In order to do this, we define a colouring on the remaining clusters of $\Sigma_{\chi}$. Each cluster can be coloured red, blue, both or neither. In particular, clusters with an odd number of red (resp. blue) children are coloured \textit{red} (resp. \textit{blue}). Clusters which are coloured \textit{both} red \textit{and} blue are \textit{purple}, and clusters which are coloured neither are \textit{black}.


From here, our description of $\Y^{\textrm{min}}_k$ is very much in the spirit of \cite[Theorem~8.5]{DDMM18} --- to a cluster $\s$ we associate $1$, $2$ or $4$ components of the special fibre, and the components of $\s$ and $\s'$ are linked by a chain of $\PP^1$s if $\s'$ is a maximal subcluster (\textit{child}) of $\s$ (or vice versa). The length of this chain is determined by $\delta_{\s'} = d_{\s'} - d_{\s}$. In Theorem \ref{thm:frob}, we also give the action of Frobenius on the the dual graph of $\Y^{\textrm{min}}_k$. Before we state our theorem, let us define some of the cluster terminology used.

\begin{definition}
	A cluster $\s\in \Sigma_{\chi}$  is \textit{odd} (resp. \textit{even}) if $|\s|$ is odd (resp. even). A \textit{child} $\s' < \s$ is a maximal subcluster of $\s$. The cluster $\s$ is \textit{\"ubereven} if all its children are even. It is \textit{proper} if $|\s| \geq 2$, \textit{principal} if $|\s| \geq 3$ (except in a few cases if $\s = \Rcal$, see Definition \ref{def:principal}) and a \textit{twin} if $|\s| = 2$. It is \textit{chromatic} if it is red, blue or purple. It has \textit{polychromatic children} if it has a purple child, or if it has both a red and a blue child (it can have other children as well) and \textit{monochromatic children} if all of its chromatic children are red or all of its chromatic children are blue. The \textit{relative depth} of a cluster $\s' < \s$ is $\delta_{\s'} = d_{\s'} - d_{\s}$ and the relative depth between any two clusters $\s_1$ and $\s_2$ is $\delta_{\s_1,\s_2} = d_{\s_1} + d_{\s_2} - 2d_{\s_1 \wedge \s_2}$ where $\s_1 \wedge \s_2$ is the smallest cluster containing both $\s_1$ and $\s_2$.
\end{definition}

\begin{definition}
	\label{def:genusintro}
	Let $\s$ be a cluster. We assign to $\s$ a \textit{chromatic genus} $g_{\chi}(\s)$ as follows: let $\s_{\chi}$ be the set of chromatic children of $\s$. If $\s$ has polychromatic children, then $g_{\chi}(\s) = |\s_{\chi}| - 3$ if $\s = \Rcal$ and $f_1$ and $f_2$ both have even degree, or $\Rcal = \s \sqcup \s'$, $\s$ is even and $f_1$ and $f_2$ both have even degree; and $g_{\chi}(\s) = |\s_{\chi}| - 2$ otherwise. If $\s$ has monochromatic children, then $g_{\chi}(\s)$ is such that $2g_{\chi}(\s) + 1 = |\s_{\chi}|$ or $2g_{\chi}(\s) + 2 = |\s_{\chi}|$. Otherwise $g_{\chi}(\s) = 0$. 
\end{definition}

If $\s$ has monochromatic children or no chromatic children then $g_{\chi}(\s) = g(\s)$, as in \cite[Section~5.3]{DDMM18}. Our main result is the following. We have excluded the cases where $\Rcal$ is not principal, but these are dealt with in Theorem \ref{thm:main}.

\begin{thm}[=Theorem \ref{thm:main}]
	\label{thm:mainintro}

		Let $K$ be a local field of odd residue characteristic and let $C_1$ and $C_2$ be two hyperelliptic curves over $K$. Let $C_h$ be their composite curve. Let $Y$ be the bihyperelliptic curve arising from $C_1$ and $C_2$, such that $Y$ has semistable reduction and all the depths in the chromatic cluster picture of $Y$ are integers. Suppose further that $\Rcal$ has at least three children, none of which of size $2g(C_h)$\footnote{This is a sufficient condition for $\Rcal$ to be principal.}. Then the dual graph of $\Y^{\mathrm{min}}_k$ has the following explicit description:
		
		Each principal cluster $\s$ contributes vertices of genus $g_{\chi}(\s)$ to the dual graph. If $\s$ is not \"ubereven: $1$ vertex $v_{\s}$ if $\s$ has polychromatic children and $2$ vertices $v_{\s}^+, v_{\s}^-$ if $\s$ has monochromatic children; and if \"ubereven: $2$ vertices $v_{\s}^+, v_{\s}^-$ if $\s$ has chromatic children and $4$ vertices $v_{\s}^{+,+}, v_{\s}^{+,-}, v_{\s}^{-,+}, v_{\s}^{-,-}$ if $\s$ has no chromatic children.
		
		These are linked by edges as follows:
		
		\begin{table}[ht]
			\centering
			\begin{tabular}{|c|c|c|c|c|}
				\hline
				Name & From & To & Length & Condition  \\ \hline
				$L_{\s'}^+$ & $v_{\s}^+$ & $v_{\s'}^{\sigma}$ & \multirow{2}{*}{$\frac{1}{2}\delta_{\s'}$} & \multirow{2}{*}{$\s' < \s$ both principal, $\s'$ chromatic} \\ \cline{1-3}
				$L_{\s'}^-$ & $v_{\s}^-$ & $v_{\s'}^{-\sigma}$ &  &  \\ \hline
				$L_{\s'}^{+,+}$ & $v_{\s}^{+,+}$ & $v_{\s'}^{+,+}$ & \multirow{4}{*}{$\delta_{\s'}$} & \multirow{4}{*}{$\s' < \s$ both principal, $\s'$ black} \\ \cline{1-3}
				$L_{\s'}^{+,-}$ & $v_{\s}^{+,-}$ & $v_{\s'}^{+,-}$ &  &  \\ \cline{1-3}
				$L_{\s'}^{-,+}$ & $v_{\s}^{-,+}$ & $v_{\s'}^{-,+}$ &  &  \\ \cline{1-3}
				$L_{\s'}^{-,-}$ & $v_{\s}^{-,-}$ & $v_{\s'}^{-,-}$ &  &  \\ \hline
				
				$L_{\tfrak}$ & $v_{\s}^+$ & $v_{\s}^-$ & $\delta_{\tfrak}$ & $\tfrak < \s$, $\tfrak$ chromatic twin, $\s$ principal \\ \hline	
				$L_{\tfrak}^+$ & $v_{\s}^{+,+}$ & $v_{\s}^{\sigma,-\sigma}$ & \multirow{2}{*}{$2\delta_{\tfrak}$} & \multirow{2}{*}{$\tfrak < \s$, $\tfrak$ black twin, $\s$ principal} \\ \cline{1-3}
				$L_{\tfrak}^-$ & $v_{\s}^{-,-}$ & $v_{\s}^{-\sigma,\sigma}$ &  & \\ \hline
			\end{tabular}
		\end{table}
		
		where $\sigma = \sigma(\s, \s')\in\{-1\}$ is defined in Definition \ref{def:gensig}; $v_{\s}^{\pm,+} = v_{\s}^{\pm,-} = v_{\s}^{\pm}$ if $\s$ is non-\"ubereven with monochromatic red children; $v_{\s}^{+,\pm} = v_{\s}^{-,\pm} = v_{\s}^{\pm}$ if $\s$ is non-\"ubereven with has monochromatic blue children; $v_{\s}^{\pm, \pm} = v_{\s}^+$, $v_{\s}^{\pm, \mp} = v_{\s}^-$ if $\s$ is \"ubereven with chromatic children; and $v_{\s}^+ = v_{\s}^- = v_{\s}$ if $\s$ is non-\"ubereven with polychromatic children.
	 
\end{thm}

\begin{remark}
The action of Frobenius can also be described explicitly in terms of the chromatic cluster picture of $Y$. See Theorem \ref{thm:frob} for details.
\end{remark}

\begin{remark}
	A condition we require is that $Y$ have semistable reduction. This is hard to check directly, but it is equivalent to the curves $C_1$, $C_2$ and $C_h$ having semistable reduction, which can be checked using the Semistability Criterion for hyperelliptic curves given in \cite[Theorem~1.8]{DDMM18}. This equivalence arises due to the existence of an isogeny between the Jacobians of $Y$ and $C_1 \times C_2 \times C_h$. It cannot be found in the literature (to our knowledge), but will appear in an upcoming paper of Dokchitser and Morgan.
\end{remark}

\begin{eg}
	\label{eg:introtwinblack}
	Let $C_1 : y_1^2 = (x^2 - p^{2n})(x-1)$ and $C_2 : y_2^2 = x^3 - 2$ be two elliptic curves with composite curve $C_h: y_h^2 = (x^2 - p^n)(x-1)(x^3 - 2)$ and associated bielliptic curve $Y$. The chromatic cluster picture of $Y$ and the dual graph of $\Y^{\textrm{min}}_k$ is below.
	
	
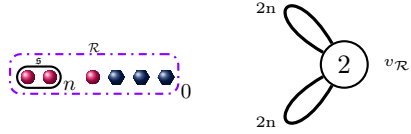
\begin{figure}[ht]
		\centering
		\clusterpicture
		\Root[A] {}{first}{r1};
		\Root[A] {}{r1}{r2};
		\ClusterLDName c1[][n][\s] = (r1)(r2);
		\Root[A] {}{c1}{r3};
		\Root[B] {}{r3}{r4};
		\Root[B] {}{r4}{r5};
		\Root[B] {}{r5}{r6};
		\ClusterLDName c2[][0][\Rcal] = (c1)(r3)(r4)(r5)(r6), clB;
		\endclusterpicture\quad \quad
		\raisebox{-3em}{\begin{tikzpicture}
		
		\tikzstyle{every node}=[font=\tiny]
		\node at (0.7, -1) {$v_{\Rcal}$};
		\begin{scope}[every node/.style={circle,thick,draw}]
		\node (A) at (0,-1) {2};
		\end{scope}
		
		\path (A) edge[very thick, out=120, in=150, looseness=16] node[left] {2n} (A);
		\path (A) edge[very thick, out=210, in=240, looseness=16] node[left] {2n} (A);  
		
		\end{tikzpicture}}
		
	\end{figure}
	
	The top cluster $\Rcal$ is not \"ubereven and has polychromatic children so contributes one component of genus $2$. The twin is black and has depth $n$ so contributes two loops of length $2n$. For a description of the action of Frobenius see Example \ref{eg:twinblack}.
\end{eg}

\begin{eg}
	\label{eg:intro}
	Let $C_1 : y^2_1 = ((x-1)^4 - p^{12})((x-2)^3 + p^6)$ and $C_2 : y_1^2 = (x^3 - p^{18})((x-2)^3 - p^6)$ be hyperelliptic curves over $K$ and let $Y$ be the associated bihyperelliptic curve. The chromatic cluster picture and dual graph of $\Y^{\textrm{min}}_k$ are shown below. The cluster picture $\Sigma_{\chi}$ consists of the top cluster $\Rcal$, a blue cluster $\s_1$, a black cluster $\s_2$ and a purple cluster $\s_3$. The top cluster has a purple child and hence there is one component, $v_{\Rcal}$, arising from $\Rcal$. The component is genus $0$ by Definition \ref{def:genusintro}. The cluster $\s_1$ has monochromatic (blue) children, so contributes two components, $v_{\s_1}^+$ and $v_{\s_1}^-$ of genus $1$. Furthermore, it is chromatic so there are two linking chains between the components of $\Rcal$ and $\s_1$, each of length $3$. The cluster $\s_2$ also has monochromatic (red) children, but it is black and hence there are four linking chains, each of length $3$, between $v_{\Rcal}$ and $v_{\s_2}^{\pm}$. Finally $\s_3$ is chromatic with polychromatic children so it contributes a single component with two linking chains to $v_{\Rcal}$. 
	
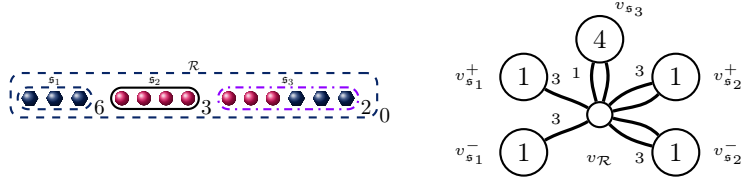
	\begin{figure}[ht]
		\centering

		\clusterpicture
		\Root[B] {} {first} {r1};
		\Root[B] {} {r1} {r2};
		\Root[B] {} {r2} {r3};
		\ClusterLDName c1[][6][\mathfrak{s}_1] = (r1)(r2)(r3),clD;
		\Root[A] {} {c1} {r4};
		\Root[A] {} {r4} {r5};
		\Root[A] {} {r5} {r6};
		\Root[A] {} {r6} {r7};
		\ClusterLDName c2[][3][\mathfrak{s}_2] = (r4)(r5)(r6)(r7);
		\Root[A] {} {c2} {r8};
		\Root[A] {} {r8} {r9};
		\Root[A] {} {r9} {r10};
		\Root[B] {} {r10} {r11};
		\Root[B] {} {r11} {r12};
		\Root[B] {} {r12} {r13};
		\ClusterLDName c3[][2][\mathfrak{s}_3] = (r8)(r9)(r10)(r11)(r12)(r13),clB;
		\ClusterLDName c4[][0][\Rcal] = (c1)(c2)(c3),clD;
		\endclusterpicture \quad \quad
		\raisebox{-3em}{\begin{tikzpicture}
		\tikzstyle{every node}=[font=\tiny]
		\node at (0, -1.6) {$v_{\Rcal}$};
		\node at (-1.7, -0.5) {$v_{\s_1}^+$};
		\node at (-1.7, -1.5) {$v_{\s_1}^-$};
		\node at (1.7, -0.5) {$v_{\s_2}^+$};
		\node at (1.7, -1.5) {$v_{\s_2}^-$};
		\node at (0.4, 0.4) {$v_{\s_3}$};
		\begin{scope}[every node/.style={circle,thick,draw}]
		\node (A) at (0, -1) {};
		\node (B) at (-1, -0.5) {1};
		\node (C) at (-1, -1.5) {1};
		\node (D) at (1, -0.5) {1};
		\node (E) at (1, -1.5) {1};
		\node (F) at (0, 0) {4};
		\end{scope}
		
		\path (A) edge[very thick, out=210, in=30] node[above, near end] {3} (C);
		\path (A) edge[very thick, out=150, in=-30] node[above, near end] {3} (B); 
		\path (A) edge[very thick, out=45, in=195] node[above, near end] {3} (D);
		\path (A) edge[very thick, out=15, in=225]  (D);
		\path (A) edge[very thick, out=-45, in=165] node[below, near end] {3} (E); 
		\path (A) edge[very thick, out=-15, in=135] (E); 
		\path (A) edge[very thick, out=105, in=-105] node[left, near end] {1} (F); 
		\path (A) edge[very thick, out=75, in=-75] (F);

		\end{tikzpicture}}
	\end{figure}
\end{eg}

\begin{eg}
	
	Let $C_1 : y_1^2 = (x^3 - p^{12})(x^3 - 2)$ and $C_2 : y_2^2 = (x^4 - p^{20})(x^3 - 1)$ be two hyperelliptic curves, $C_h$ their composite hyperelliptic curve and $Y$ the associated bihyperelliptic curve. The chromatic cluster picture of $Y$ and the dual graph of $\Y^{\textrm{min}}_k$ are below.
	
	
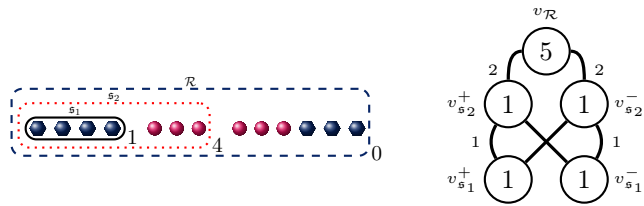
\begin{figure}[ht]
		\centering
		\clusterpicture
		\Root[B] {}{first}{r1};
		\Root[B] {}{r1}{r3};
		\Root[B] {}{r3}{r4};
		\Root[B] {}{r4}{r5};
		\ClusterLDName c1[][1][\s_1] = (r1)(r3)(r4)(r5);
		\Root[A] {}{c1}{r6};
		\Root[A] {}{r6}{r7};
		\Root[A] {}{r7}{r8};
		\ClusterLDName c2[][4][\s_2] = (c1)(r6)(r7)(r8), clC;
		\Root[A] {}{c2}{r9};
		\Root[A] {}{r9}{r10};
		\Root[A] {}{r10}{r11};
		\Root[B] {}{r11}{r12};
		\Root[B] {}{r12}{r13};
		\Root[B] {}{r13}{r14};
		\ClusterLDName c3[][0][\mathcal{R}] = (c2)(r9)(r10)(r11)(r12)(r13)(r14), clD;
		\endclusterpicture \qquad
		\raisebox{-3em}{\begin{tikzpicture}
		\tikzstyle{every node}=[font=\tiny]
		\node at (0, -0.8) {$v_{\Rcal}$};
		\node at (-1.1, -3) {$v_{\s_1}^+$};
		\node at (1.1, -3) {$v_{\s_1}^-$};
		\node at (-1.1, -2) {$v_{\s_2}^+$};
		\node at (1.1, -2) {$v_{\s_2}^-$};
		
		\begin{scope}[every node/.style={circle,thick,draw}]
		\node (A) at (0,-1.3) {5};
		\node (B) at (-0.5, -2) {1};
		\node (C) at (0.5, -2) {1};
		\node (D) at (-0.5, -3) {1};
		\node (E) at (0.5, -3) {1};
		\end{scope}
		
		\path (A) edge[very thick, out=0, in=90] node[right, near end] {2} (C);
		\path (A) edge[very thick, out=180, in=90] node[left, near end] {2} (B);  
		\path (B) edge[very thick, out=240, in=120] node[left] {1} (D);
		\path (C) edge[very thick, out=300, in=60] node[right] {1} (E);
		\path (B) edge[very thick, out=315, in=135] (E);
		\path (C) edge[very thick, out=225, in=45] (D);

		\end{tikzpicture}}
	\end{figure}
\end{eg}

\begin{remark}
	A perhaps more suitable notion of a bihyperelliptic curve would be any curve $Y$ which has a degree $2$ map to a hyperelliptic curve $C$, directly mirroring the definition for bielliptic curves. However, the cover $Y/\PP^1$ is not necessarily Galois in this case. Since we require  this for our purposes, we shall restrict to the case given in Definition \ref{def:bihyp}, but the more general case is also of interest.
\end{remark}

\subsection{Structure} In Section \ref{sec:ccp} we define chromatic cluster pictures, the main object we will use in our theorem to describe the minimal regular model of a bihyperelliptic curve. Some familiarity with cluster pictures is helpful: for a comprehensive definition we direct the reader to \cite{DDMM18}, and for many examples of cluster pictures in use to \cite{best2020users}. In Section \ref{sec:results} we state our main Theorem \ref{thm:main} and illustrate it with several examples. The proof of Theorem \ref{thm:main} is Section \ref{sec:proof}. The strategy is to normalise a model $\X$ of $\PP^1$ associated to $\Sigma_{\chi}$ in the function field of $Y$, resulting in a model $\Y$ of $Y$ over $\Kur$. Using the results of \cite{DDMM18}, we deduce the normalisation of $\X$ in the  function fields of $C_1$, $C_2$ and $C_h$ and use these to give an explicit description of $\Y$.

\subsection{Notation} Throughout $K$ will be a local field with ring of integers $\OO_K$ and residue field $k$ of characteristic $p > 2$. The uniformiser of $K$ will be $\pi$, and the valuation $v_K$. A fixed algebraic closure of $K$ will be denoted $\overline{K}$ and $\Kur$ will be the maximal unramified extension of $K$. For us, a \textit{model} $\mathcal{C}$ of a curve $C$ over $K$ is a flat proper $\mathcal{O}_K$-scheme whose generic fibre is isomorphic to $C$. Furthermore, this model is \textit{regular} if $\mathcal{C}$ is regular as a scheme, it is \textit{minimal} if it is dominates no other regular model, and it is \textit{semistable} if the special fibre is reduced with at worst normal crossings such that any component isomorphic to $\PP^1$ intersects the rest of the special fibre in at least two points.

In cluster pictures, red roots will be represented by spheres \raisebox{0.5em}{\clusterpicture \Root[A] {}{first}{r1}; \endclusterpicture} and blue roots by hexagons \raisebox{0.5em}{\clusterpicture \Root[B] {}{first}{r1}; \endclusterpicture}. Red clusters will be denoted with dotted lines, blue clusters by dashed lines, purple cluster by dot-dash lines and black clusters by solid lines. Sometimes, we will need an empty cluster, which will look like this: \raisebox{0.5em}{\clusterpicture \Root[] {}{first}{r1} \ClusterLDName c1[][][] = (r1); \endclusterpicture}. The relative depth of a cluster $\s \neq \Rcal$ will be written to the lower right of $\s$, and the depth of $\Rcal$ will be written to its lower right. The name of the cluster is above it.

Special fibres will be draw as their dual graphs. Components from principal clusters will be the vertices, with their genus inside. If no genus is written, the corresponding component has genus $0$. The edges between nodes are chains of rational curves, and the number next to them indicates the number of edges in the dual graph corresponding to the linking chain (which is one more than the number of $\PP^1s$ in the linking chain). An edge of length $1$ in the dual graph indicates an intersection between the components. When there are two edges between a pair of vertices, only one edge will be labelled with its length, as the other edge will always have the same length.

\subsection*{Acknowledgements} I would like to thank my supervisor Vladimir Dokchitser for his continued advice and support. I would also like to thank Holly Green for checking and correcting many earlier versions of the theorem, and Adam Morgan for his patience answering my sporadic questions. This work was supported by the Engineering and Physical
Sciences Research Council [EP/L015234/1], the EPSRC Centre for Doctoral Training in Geometry
and Number Theory (The London School of Geometry and Number Theory), and University College
London.

\section{Cluster Pictures}
\label{sec:ccp}

\subsection{The Story so Far}

A cluster picture is a set of subsets (``clusters'') of the roots of a polynomial $f$, where each cluster contains roots that are $p$-adically close together. There are many numerical invariants associated to cluster picture which are used, and in this section we recall those that we will require. Throughout this subsection, $C : y^2 = f(x) = c_f \prod_{r \in \Rcal} (x-r)$ will be a hyperelliptic curve of genus $g \geq 2$ with semistable reduction and cluster picture $\Sigma$, as in Definition \ref{def:cluster}. Denote by $\mathcal{C}$ the minimal regular model of $C$ as described in \cite[Theorem~8.5]{DDMM18}. 

The cluster picture $\Sigma$ is equivalent to an \textit{admissible collection of disks}, that is a non empty set $\mathcal{D}$ of disks $D = D_{z,d} = \{ x \in \overline{K}\, |\, v(x - z_D) \geq d_D\}$ such that $z_D \in \Kur$ (called \textit{integral disks}), $d_D \in \Z$, where $\mathcal{D}$ has a maximal element with respect to inclusion and such that if $D_1 \subseteq D_2$ then any integral disk with $D_1 \subseteq D \subseteq D_2$ is in $\mathcal{D}$. This collection of disks in turn gives rise to a model $\X$ of $\PP^1$ which separates the branch points of $C \rightarrow \PP^1$. For the precise construction, see \cite[Definition~3.8]{DDMM18}. To each disk in such a collection we define a valuation extending $v_K$.

\begin{definition}
	Let $D$ be a $p$-adic disk. Then $$\nu_D(f) = v_K(c_f) + \sum_{r \in \Rcal} \min(d_D, v_K(z_D - r)),$$ and $\omega_D(f) \in \{0, 1\}$ is such that $\omega_D(f) \equiv v_D(f) \mod 2$.
\end{definition}

Each proper cluster $\s$ is given a genus, which is the genus of the component in $\mathcal{C}$ arising from $\s$.

\begin{definition}
	Let $\s \in \Sigma$ be a proper cluster. Then the \textit{genus} of $\s$, $g(\s)$ is such that $2g(\s) + 1$ or $2g(\s) + 2$ equals the number of odd children of $\s$.
\end{definition}

To each even cluster or \textit{cotwin}, a cluster who has a child of size $2g$ whose complement is not a twin, we assign a character of $\Gal(\overline{K}/K)$, which roughly tell us whether or not $\sigma \in \Gal(\overline{K}/K)$ swaps the points at infinity of the component arising from $\s$.

\begin{definition}
	Let $\s$ be an even cluster or a cotwin. We define the character $\epsilon_{\s} : \Gal(\overline{K}/K) \rightarrow \{ \pm 1\}$ by $\epsilon_{\s}(\sigma) = \sigma(\theta_{\s^*}) / \theta_{\sigma(\s^*)}$, where for any cluster $\tfrak$, $\theta_{\tfrak}$ is a choice of square root $\sqrt{c_f \prod_{r \not \in \tfrak} (z_{\tfrak} - r)}$ and $\tfrak^*$ is the smallest cluster containing $\tfrak$ whose parent is not \"ubereven (or $\Rcal$ if no such cluster exists).
\end{definition}

Note that the definition above extends to ``empty clusters'' --- $p$-adic disks which contain no roots. This will be useful when talking about red and blue cluster pictures in the next section, which can contain empty clusters.

\subsection{Chromatic Cluster Pictures}

The combinatorial objects we will use to calculate the semistable model will be \textit{chromatic cluster pictures}.  A chromatic cluster picture is similar to a cluster picture, but instead of one polynomial we take a product of coprime polynomials $f_1 f_2$, and we colour the roots red and blue to indicate whether they are roots of $f_1$ or $f_2$. In this way, the chromatic cluster picture contains the information of the cluster pictures associated to $f_1$, $f_2$ and $f_1 f_2$.

\begin{definition}
	A \textit{chromatic cluster picture} $\Sigma_{\chi}$ is a cluster picture $\Sigma$ on a set $\Rcal$ with a colouring function $c: \Rcal \longrightarrow \{ \textrm{red}, \textrm{ blue}\}$, assigning to each root a colour.
	
	This induces a colouring on the remaining clusters as follows:
	
	\begin{enumerate}
		\item clusters with an odd number of blue children and an even number of red children (resp. an odd number of red children and an even number of blue children) are coloured blue (resp. red),
		\item clusters with an odd number of blue children and an odd number of red children are coloured purple,
		\item all other clusters are coloured black.
	\end{enumerate}
	
	where purple children are included in both the set of red \textit{and} blue clusters. Blue, red and purple clusters are called \textit{chromatic clusters}. 
	
	Clusters with purple children, or clusters with both blue and red children have \textit{polychromatic children}, whereas clusters whose only chromatic children are red or blue have \textit{monochromatic children}.
	
	We define the \textit{red} (resp. \textit{blue}) cluster picture $\Sigma_1$ (resp. $\Sigma_2$) associated to $\Sigma_{\chi}$ to be the subset of $\Sigma$ where the only clusters of size $1$ are the red (resp. blue) ones. We forget the colouring on the rest of the clusters.
\end{definition}

\begin{lemma}
	\label{lem:colparity}
	Let $\s \in \Sigma$ be a cluster. If $\s$ is odd then $\s$ is red or blue, and if $\s$ is even then $\s$ is purple or black. Furthermore, purple clusters are odd in $\Sigma_1$ and $\Sigma_2$ and even in $\Sigma$, red clusters are odd in $\Sigma_1$ and $\Sigma$ and even in $\Sigma_2$, blue clusters are odd in $\Sigma_2$ and $\Sigma$ and even in $\Sigma_1$, and finally black clusters are even in all of $\Sigma_1, \Sigma_2$ and $\Sigma$.
\end{lemma}
\begin{proof}
	This follows by induction on the size of the cluster, noting that it is trivially true for singletons.
\end{proof}

\begin{remark}
	The red and blue cluster pictures aren't cluster pictures (or chromatic cluster pictures) in the conventional sense --- they are cluster pictures with some additional $p$-adic disks. They can have ``empty'' clusters which contain no roots, or clusters which only contain another cluster and nothing else. However, they give rise to the same \textit{admissible collection of discs}, in the sense of \cite[Definition~3.4]{DDMM18}, as $\Sigma_h$. Therefore we can apply the results of \cite{DDMM18} to them. 
\end{remark}

We will be interested in the chromatic cluster pictures of bihyperelliptic curves. In other words, if $Y : \{y_1^2 = f_1, y_2^2 = f_2\}$ is a bihyperelliptic curve, then its chromatic cluster picture arises from colouring the roots of $f_1$ red, the roots of $f_2$ blue and the rest of the clusters according to the rules (ii) - (iv).

\begin{eg}
	
	Below (left) is the chromatic cluster picture of $Y:\{y_1^2 = (x-p^n)(x^2-1), y_2^2 = (x+p^n)(x^2-2)\}$. On the right is the red cluster picture, which is equivalent to the blue cluster picture.
	
	
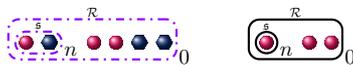
\begin{figure}[ht]
		\centering
		\clusterpicture
		\Root[A] {}{first}{r1};
		\Root[B] {}{r1}{r2};
		\ClusterLDName c1[][n][\s] = (r1)(r2), clB;
		\Root[A] {}{c1}{r3};
		\Root[A] {}{r3}{r4};
		\Root[B] {}{r4}{r5};
		\Root[B] {}{r5}{r6};
		\ClusterLDName c2[][0][\Rcal] = (c1)(r3)(r4)(r5)(r6), clB;
		\endclusterpicture \qquad
		\clusterpicture
		\Root[A] {}{first}{r1};
		\ClusterLDName c1[][n][\s] = (r1);
		\Root[A] {}{c1}{r3};
		\Root[A] {}{r3}{r4};
		\ClusterLDName c2[][0][\Rcal] = (c1)(r3)(r4);
		\endclusterpicture

	\end{figure}

	The cluster $\s\in \Sigma_{\chi}$ contains both an odd number of red children and an odd number of blue children (one of each), and hence is purple. The cluster $\Rcal \in \Sigma_{\chi}$ contains three red children (since the purple child counts as a red and a blue child) and three blue children and hence is also purple.
	
	The red and blue clusters are equivalent, and in both the ``cluster'' $\s$ is a $p$-adic disk containing a single root. This wouldn't be a cluster in a conventional cluster picture.
\end{eg}

\begin{eg}
	From left to right, we have the chromatic, red and blue cluster pictures of $Y:\{y_1^2 = (x^2 - p^n)(x-1), y_2^2 = x^3 - 2\}$. The cluster $\s \in \Sigma_{\chi}$ has an even number of both red and blue children and hence is black. On the other hand, $\Rcal$ has an odd number of both red and blue children and hence is coloured purple. Note that in the blue cluster picture $\s$ is an empty cluster (i.e. it is a $p$-adic disk which contains no roots).
	
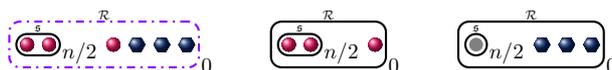
\begin{figure}[ht]
		\centering
\clusterpicture
\Root[A] {}{first}{r1};
\Root[A] {}{r1}{r2};
\ClusterLDName c1[][n/2][\s] = (r1)(r2);
\Root[A] {}{c1}{r3};
\Root[B] {}{r3}{r4};
\Root[B] {}{r4}{r5};
\Root[B] {}{r5}{r6};
\ClusterLDName c2[][0][\Rcal] = (c1)(r3)(r4)(r5)(r6), clB;
\endclusterpicture \qquad
\clusterpicture
\Root[A] {}{first}{r1};
\Root[A] {}{r1}{r2};
\ClusterLDName c1[][n/2][\s] = (r1)(r2);
\Root[A] {}{c1}{r3};
\ClusterLDName c2[][0][\Rcal] = (c1)(r3);
\endclusterpicture \qquad
\clusterpicture
\Root[D] {}{first}{r1};
\ClusterLDName c1[][n/2][\s] = (r1);
\Root[B] {}{c1}{r4};
\Root[B] {}{r4}{r5};
\Root[B] {}{r5}{r6};
\ClusterLDName c2[][0][\Rcal] = (c1)(r4)(r5)(r6);
\endclusterpicture \qquad

	\end{figure}
\end{eg}

\begin{eg}
	From left to right, the chromatic, red, and blue cluster picture of the bihyperelliptic curve $Y:\{y_1^2 = (x^3 - p^9)(x^3 - 2), y_2^2 = (x^4 - p^{20})(x^3 - 1)\}$. In the red cluster picture $\s_1$ is an empty cluster, and in the blue cluster picture $\s_2$ has a unique child. Neither of these are traditionally clusters.
	
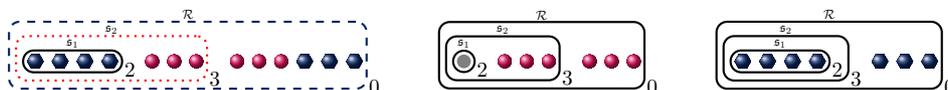
\begin{figure}[ht]
		\centering

	\clusterpicture
	\Root[B] {}{first}{r1};
	\Root[B] {}{r1}{r3};
	\Root[B] {}{r3}{r4};
	\Root[B] {}{r4}{r5};
	\ClusterLDName c1[][2][\s_1] = (r1)(r3)(r4)(r5);
	\Root[A] {}{c1}{r6};
	\Root[A] {}{r6}{r7};
	\Root[A] {}{r7}{r8};
	\ClusterLDName c2[][3][\s_2] = (c1)(r6)(r7)(r8), clC;
	\Root[A] {}{c2}{r9};
	\Root[A] {}{r9}{r10};
	\Root[A] {}{r10}{r11};
	\Root[B] {}{r11}{r12};
	\Root[B] {}{r12}{r13};
	\Root[B] {}{r13}{r14};
	\ClusterLDName c3[][0][\Rcal] = (c2)(r9)(r10)(r11)(r12)(r13)(r14), clD;
	\endclusterpicture \qquad
	\clusterpicture
	\Root[] {}{first}{r1};
	\ClusterLDName c1[][2][\s_1] = (r1);
	\Root[A] {}{c1}{r6};
	\Root[A] {}{r6}{r7};
	\Root[A] {}{r7}{r8};
	\ClusterLDName c2[][3][\s_2] = (c1)(r6)(r7)(r8);
	\Root[A] {}{c2}{r9};
	\Root[A] {}{r9}{r10};
	\Root[A] {}{r10}{r11};
	\ClusterLDName c3[][0][\Rcal] = (c2)(r9)(r10)(r11);
	\endclusterpicture \qquad
	\clusterpicture
	\Root[B] {}{first}{r1};
	\Root[B] {}{r1}{r3};
	\Root[B] {}{r3}{r4};
	\Root[B] {}{r4}{r5};
	\ClusterLDName c1[][2][\s_1] = (r1)(r3)(r4)(r5);
	\ClusterLDName c2[][3][\s_2] = (c1);
	\Root[B] {}{c2}{r12};
	\Root[B] {}{r12}{r13};
	\Root[B] {}{r13}{r14};
	\ClusterLDName c3[][0][\Rcal] = (c2)(r12)(r13)(r14);
	\endclusterpicture \qquad
	\end{figure}
\end{eg}

The following is an important definition, the clusters which will contribute principal components to the semistable model.

\begin{definition}
	\label{def:principal}
	Let $\s$ be a cluster. Then $\s$ is \textit{chromatically principal} if $|\s| \geq 3$, except in the following cases:
	\begin{enumerate}
		\item $\s = \Rcal = \s_1 \sqcup \s_2$, with $\s_1$ and $\s_2$ of the same colour, or one of $\s_1$ or $\s_2$ a singleton or a twin,
		\item $\s = \Rcal$ has a unique proper child $\s'$ of size $2g(C_h)$ such that either $\s'$ is purple or $\s$ and $\s'$ are both black.
	\end{enumerate}
\end{definition}

We can characterise this as follows:

\begin{lemma}
	A cluster $\s$ is chromatically principal if and only if it is principal in at least one of $\Sigma_1, \Sigma_2, \Sigma_h$.
\end{lemma}
\begin{proof}
	If $\s \neq \Rcal$ this is clear as its size in $\Sigma_{\chi}$ is the same as its size in $\Sigma_h$. Therefore, suppose $\s = \Rcal$. Suppose $\s = \s_1 \sqcup \s_2$ with $\s_1$ and $\s_2$ the same colour. Then $\s_1$ and $\s_2$ have the same parity in each of $\Sigma_1, \Sigma_2, \Sigma_h$ by Lemma \ref{lem:colparity}. If $\s$ has a unique proper child of size $2g(C_h)$ then $\Sigma_1$ has a child of size $2g(C_1)$, and similarly for $\Sigma_2$ and $\Sigma_h$. The other direction can be checked similarly (for example, if $\s = \s_1 \sqcup \s_2$ with $\s_1$, $\s_2$ principal of different colours, then $\s_1$ and $\s_2$ will have opposite parities in at least one of $\Sigma_1$, $\Sigma_2$ or $\Sigma_h$).
\end{proof}

We must update our definition of cotwin to exclude a case which cannot happen for classic cluster pictures:

\begin{definition}
	A cluster $\s$ is a \textit{cotwin} if it has a child $\s'$ of size $2g(C_h)$ whose complement isn't a twin, unless $\s$ is purple and $\s'$ is black.
\end{definition}

We finish the section with a definition which will be needed in stating our main theorem.


\begin{definition}
	\label{def:gensig}
	If $\s$ is a cluster then $\s_{\chi}$ is the set of chromatic children of $\s$. If $\s, \s'$ are two proper clusters then $\sigma(\s, \s') = -1$ in the following cases:
	\begin{enumerate}
		\item $\s$ and $\s'$ each have monochromatic children of opposite colours,
		\item $\s$ is black \"ubereven and $\s'$ is black with monochromatic, blue children;
	\end{enumerate}
	and $1$ otherwise.
\end{definition}

\section{Statement of Results}
\label{sec:results}

Here we describe our main theorem: how to construct the minimal regular model $\Y^{\min}$ of $Y$ given the chromatic cluster picture of $Y$. The rough idea is that principal clusters give rise to components in the special fibre, and components of $\s$ are linked to the components of the children of $\s$. The number of components and linking chains, as well as the lengths of the linking chains, are determined by the properties of the clusters. 

\begin{thm}
	\label{thm:main}
	Let $K$ be a local field of odd residue characteristic and let $C_1$ and $C_2$ be two hyperelliptic curves over $K$. Let $C_h$ be their composite curve. Let $Y$ be the bihyperelliptic curve arising from $C_1$ and $C_2$, such that $Y$ has semistable reduction and all the depths in the chromatic cluster picture of $Y$ are integers. Then the dual graph of $\Y^{\mathrm{min}}_k$ is entirely determined by the chromatic cluster picture of $Y$.
	
	In particular, each principal cluster $\s$ contributes vertices of genus $g_{\chi}(\s)$ to the dual graph of $\Y^{\textrm{min}}_k$. If $\s$ is not \"ubereven: $1$ vertex $v_{\s}$ if $\s$ has polychromatic children and $2$ vertices $v_{\s}^+, v_{\s}^-$ if $\s$ has monochromatic children; and if \"ubereven: $2$ vertices $v_{\s}^+, v_{\s}^-$ if $\s$ has chromatic children and $4$ vertices $v_{\s}^{+,+}, v_{\s}^{+,-}, v_{\s}^{-,+}, v_{\s}^{-,-}$ if $\s$ has no chromatic children.
	
	These are linked by edges as follows:
	
	\begin{table}[ht]
		\centering
		\begin{tabular}{|c|c|c|c|c|}
			\hline
			Name & From & To & Length & Condition  \\ \hline
			$L_{\s'}^+$ & $v_{\s}^+$ & $v_{\s'}^{\sigma}$ & \multirow{2}{*}{$\frac{1}{2}\delta_{\s'}$} & \multirow{2}{*}{$\s' < \s$ both principal, $\s'$ chromatic} \\ \cline{1-3}
			$L_{\s'}^-$ & $v_{\s}^-$ & $v_{\s'}^{-\sigma}$ &  &  \\ \hline
			$L_{\s'}^{+,+}$ & $v_{\s}^{+,+}$ & $v_{\s'}^{+,+}$ & \multirow{4}{*}{$\delta_{\s'}$} & \multirow{4}{*}{$\s' < \s$ both principal, $\s'$ black} \\ \cline{1-3}
			$L_{\s'}^{+,-}$ & $v_{\s}^{+,-}$ & $v_{\s'}^{+,-}$ &  &  \\ \cline{1-3}
			$L_{\s'}^{-,+}$ & $v_{\s}^{-,+}$ & $v_{\s'}^{-,+}$ &  &  \\ \cline{1-3}
			$L_{\s'}^{-,-}$ & $v_{\s}^{-,-}$ & $v_{\s'}^{-,-}$ &  &  \\ \hline
			
			$L_{\tfrak}$ & $v_{\s}^+$ & $v_{\s}^-$ & $\delta_{\tfrak}$ & $\tfrak < \s$, $\tfrak$ chromatic twin, $\s$ principal \\ \hline	
			$L_{\tfrak}^+$ & $v_{\s}^{+,+}$ & $v_{\s}^{\sigma,-\sigma}$ & \multirow{2}{*}{$2\delta_{\tfrak}$} & \multirow{2}{*}{$\tfrak < \s$, $\tfrak$ black twin, $\s$ principal} \\ \cline{1-3}
			$L_{\tfrak}^-$ & $v_{\s}^{-,-}$ & $v_{\s}^{-\sigma,\sigma}$ &  & \\ \hline
		\end{tabular}
	\end{table}
	
	where $\sigma = \sigma(\s, \s')$; $v_{\s}^{\pm,+} = v_{\s}^{\pm,-} = v_{\s}^{\pm}$ if $\s$ is non-\"ubereven with monochromatic red children; $v_{\s}^{+,\pm} = v_{\s}^{-,\pm} = v_{\s}^{\pm}$ if $\s$ is non-\"ubereven with has monochromatic blue children; $v_{\s}^{\pm, \pm} = v_{\s}^+$, $v_{\s}^{\pm, \mp} = v_{\s}^-$ if $\s$ is \"ubereven with chromatic children; and $v_{\s}^+ = v_{\s}^- = v_{\s}$ if $\s$ is non-\"ubereven with polychromatic children.
	
	\pagebreak
	
	Moreover, if $\Rcal$ is not principal\footnote{Recall that if $\Rcal$ is purple and has a unique proper child $\s'$ of size $2g(C_h)$ which is black, then $\Rcal$ is principal and doesn't fall in this category.}, then there are the following additional edges:
	
	\begin{table}[ht]
		\centering
		\begin{tabular}{|c|c|c|c|c|}
			\hline
			Name & From & To & Length & Condition  \\ \hline
			$L_{\tfrak}$ & $v_{\s}^+$ & $v_{\s}^-$ & $\delta_{\s}$ & $\s < \tfrak$, $\tfrak$ cotwin, $\s$ purple \\ \hline	
			$L_{\tfrak}^+$ & $v_{\s}^{+,+}$ & $v_{\s}^{\sigma,-\sigma}$ & \multirow{2}{*}{$2\delta_{\s}$} & \multirow{2}{*}{$\s < \tfrak$, $\tfrak$ cotwin, $\s$ black} \\ \cline{1-3}
			$L_{\tfrak}^-$ & $v_{\s}^{-,-}$ & $v_{\s}^{-\sigma, \sigma}$ &  &  \\ \hline	
			
			$L_{\s,\s'}^+$ & $v_{\s}^+$ & $v_{\s'}^+$ & \multirow{2}{*}{$\frac12 (\delta_{\s} + \delta_{\s'})$} & \multirow{2}{*}{$\Rcal = \s \sqcup \s'$, $\s, \s'$ both principal, same chromatic colour} \\ \cline{1-3}
			$L_{\s,\s'}^-$ & $v_{\s}^-$ & $v_{\s'}^-$ &  &  \\ \hline
			
			$L_{\s,\s'}^{+,+}$ & $v_{\s}^{+,+}$ & $v_{\s'}^{+,+}$ & \multirow{4}{*}{$\delta_{\s} + \delta_{\s'}$} & \multirow{4}{*}{$\Rcal = \s \sqcup \s'$, $\s, \s'$ both principal, black} \\ \cline{1-3}
			$L_{\s,\s'}^{+,-}$ & $v_{\s}^{+,-}$ & $v_{\s'}^{+,-}$ &  &  \\ \cline{1-3}
			$L_{\s,\s'}^{-,+}$ & $v_{\s}^{-,+}$ & $v_{\s'}^{-,+}$ &  &  \\ \cline{1-3}
			$L_{\s,\s'}^{-,-}$ & $v_{\s}^{-,-}$ & $v_{\s'}^{-,-}$ &  &  \\ \hline
			
			$L_{\tfrak}$ & $v_{\s}^+$ & $v_{\s}^-$ & $\delta_{\s} + \delta_{\tfrak}$ & $\Rcal = \s \sqcup \tfrak$, $\s$ principal, $\tfrak$ twin, both purple \\ \hline
			
			$L_{\tfrak}^+$ & $v_{\s}^{+,+}$ & $v_{\s}^{\sigma,-\sigma}$ & \multirow{2}{*}{$2(\delta_{\s} + \delta_{\tfrak})$} & \multirow{2}{*}{$\Rcal = \s \sqcup \tfrak$, $\s$ principal, $\tfrak$ twin, both black} \\ \cline{1-3}
			$L_{\tfrak}^-$ & $v_{\s}^{-,-}$ & $v_{\s}^{-\sigma,\sigma}$ &  & \\ \hline	
		\end{tabular}
	\end{table}
\end{thm}

\begin{remark}
	\label{rmk:twincomps}
	Possibly after a totally ramified extension, we can still think of proper, not principal clusters as contributing components to the special fibre. However, these components may not be principal In other words, they contribute components isomorphic to $\PP^1$ which intersects the rest of the special fibre in precisely two places. So for example, a loop $L_{\tfrak}$ arising from a chromatic twin can be thought of as a component $v_{\tfrak}$ with two linking chains $L_{\tfrak}^+, L_{\tfrak}^-$ to $v_{P(\tfrak)}$ of length $ \frac12 \delta_{\tfrak}$ (i.e., the linking chains arising in the first two rows of the first table). We have chosen not to state our theorems in this way since the component $v_{\tfrak}$ sometimes only appears in the minimal regular model after a totally ramified extension (and has multiplicity $2$ otherwise), but for the purposes of our proof we will usually take this point of view. That is, we will go to a totally ramified extension, treat all proper clusters as if they give us components, and then use Lemma \ref{lem:ramextmodel} to move between totally ramified extensions.
\end{remark}

\begin{thm}
	\label{thm:frob}
	Denote the Frobenius automorphism by $\Frobtwo$. It acts on $\Y^{\min}_k$ in the following way:
	
	\begin{enumerate}
		\item $\Frobtwo(\Gamma_{\s}^{\pm}) = \Gamma_{\Frobtwo(\s)}^{\pm\epsilon_{\s,i}(\Frobtwo)}$ \textrm{for $\s$ with chromatic children, $i \in \{1,2,h\}$ with $\s \in \Sigma_i$ \"ubereven},
		\item $\Frobtwo(\Gamma_{\s}^{\pm,\pm}) = \Gamma_{\Frobtwo(\s)}^{\pm\epsilon_{\s,2}(\Frobtwo),\pm\epsilon_{\s,1}(\Frobtwo)}$ \textrm{for $\s$ \"ubereven with no chromatic children},
		\item $\Frobtwo(L_{\s}^{\pm}) = \Gamma_{\Frobtwo(\s)}^{\pm\epsilon_{\s,i}(\Frobtwo)}$ \textrm{for $\s$ chromatic, $i \in \{1,2,h\}$ with with $\s \in \Sigma_i$ even},
		\item $\Frobtwo(L_{\s}^{\pm,\pm}) = L_{\Frobtwo(\s)}^{\pm\epsilon_{\s,2}(\Frobtwo),\pm\epsilon_{\s,1}(\Frobtwo)}$ \textrm{for $\s$ black},
		\item $\Frobtwo(L_{\tfrak}) = \epsilon_{\tfrak, h}(\Frobtwo) L_{\Frobtwo(\tfrak)}$ for $\tfrak$ a chromatic twin, where $-L$ denotes $L$ with the opposite orientation,
		\item $\Frobtwo(L_{\tfrak}^{\pm}) = \epsilon_{\tfrak, j}(\Frobtwo) L_{\Frobtwo(\tfrak)}^{\pm \epsilon_{\tfrak, i}(\Frobtwo)}$ for $\tfrak$ a black twin, $i,j \in \{1, 2\}$ such that $\tfrak$ is empty in $\Sigma_i$ and $i \neq j$.
	\end{enumerate}
\end{thm}


\begin{remark}
	For simplicity of proof, we have added the technical condition that the depth of all clusters are integers, even though there exist semistable curves with clusters of half integer depth (for example, any bihyperelliptic curve where $C_1$ is the elliptic curve given by $y_1^2 = (x^2 - p^3)(x-1))$. However, this is only a very mild restriction, since this condition can always be attained by going to the ramified extension of $K$ of degree $2$. The dual graph of $\Y_k$ over $K$ is then the same as the dual graph of $\Y_k$ over $K(\sqrt{\pi})$, except the lengths of all linking chains are halved (see Lemma \ref{lem:ramextmodel}).
\end{remark}


\begin{eg}
	\label{eg:twinchrom}
	Recall example \ref{eg:twinintrochrom}. Let $C_1: y_1^2 = (x-p^n)(x^2 - 1)$ and $C_2 : y_2^2 = (x+p^n)(x^2 - 2)$ be elliptic curves, with composite curve $C_h: y_h^2 = (x^2 - p^{2n})(x^2 - 1)(x^2 - 2)$ and associated bielliptic curve $Y$. The chromatic cluster picture of $Y$ and the dual graph of $\Y^{\textrm{min}}_k$ are below.
	\begin{figure}[ht]
		\centering
		\clusterpicture
		\Root[A] {}{first}{r1};
		\Root[B] {}{r1}{r2};
		\ClusterLDName c1[][n][\s] = (r1)(r2), clB;
		\Root[A] {}{c1}{r3};
		\Root[A] {}{r3}{r4};
		\Root[B] {}{r4}{r5};
		\Root[B] {}{r5}{r6};
		\ClusterLDName c2[][0][\Rcal] = (c1)(r3)(r4)(r5)(r6), clB;
		\endclusterpicture \quad \quad
		\raisebox{-1.8em}{\begin{tikzpicture}
			\tikzstyle{every node}=[font=\tiny]
			\node at (0.7, 0) {$v_{\Rcal}$};
			\begin{scope}[every node/.style={circle,thick,draw}]
			\node (A) at (0,0) {3};
			\end{scope}
			
			\path (A) edge[very thick, out=150, in=210, looseness=8] node[left] {n} (A);
			
			\end{tikzpicture}}
	\end{figure}
	
	In order to understand the Frobenius automorphism $\phi$, we note that the twin $\s$ is only even in $\Sigma_h$, not $\Sigma_1$ or $\Sigma_2$, so the action of Frobenius $\phi$ on the loop is $\phi(L_{\s}) = \epsilon_{\s,h}(\phi)L_{\s}$. We can calculate $\epsilon_{\s,h} = \left( \frac2p \right)$ and so the loop is inverted if and only if $2$ is not a quadratic residue mod $p$.
\end{eg}

\begin{eg}
	\label{eg:twinblack}
	Recall example \ref{eg:introtwinblack}. Let $C_1 : y_1^2 = (x^2 - p^{2n})(x-1)$ and $C_2 : y_2^2 = x^3 - 2$ be two elliptic curves with composite curve $C_h: y_h^2 = (x^2 - p^n)(x-1)(x^3 - 2)$ and associated bielliptic curve $Y$. The chromatic cluster picture of $Y$ and the dual graph of $\Y^{\textrm{min}}_k$ are below.
	
	
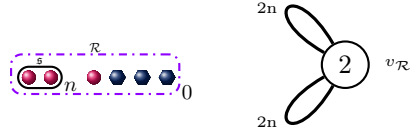
\begin{figure}[ht]
		\centering
		\clusterpicture
		\Root[A] {}{first}{r1};
		\Root[A] {}{r1}{r2};
		\ClusterLDName c1[][n][\s] = (r1)(r2);
		\Root[A] {}{c1}{r3};
		\Root[B] {}{r3}{r4};
		\Root[B] {}{r4}{r5};
		\Root[B] {}{r5}{r6};
		\ClusterLDName c2[][0][\Rcal] = (c1)(r3)(r4)(r5)(r6), clB;
		\endclusterpicture\quad \quad
		\raisebox{-3em}{\begin{tikzpicture}
			
			\tikzstyle{every node}=[font=\tiny]
			\node at (0.7, -1) {$v_{\Rcal}$};
			\begin{scope}[every node/.style={circle,thick,draw}]
			\node (A) at (0,-1) {2};
			\end{scope}
			
			\path (A) edge[very thick, out=120, in=150, looseness=16] node[left] {2n} (A);
			\path (A) edge[very thick, out=210, in=240, looseness=16] node[left] {2n} (A);  
			
			\end{tikzpicture}}
		
	\end{figure}
	
	Frobenius, $\phi$ acts on the loops of $\s$ as $\Frobtwo(L_{\s}^{\pm}) = \epsilon_{\s,1} L_{\s}^{\pm \epsilon_{\s,2}(\phi)}$, since $\s$ is empty in $\Sigma_2$. We can calculate $\epsilon_{\s,1}(\Frobtwo) = \left( \frac{-1}{p} \right)$ and $\epsilon_{\s,2} = \left( \frac{-2}{p} \right)$, and so the action of Frobenius swaps the loops if $-2$ is not a quadratic residue mod $p$, and the loops are inverted if $-1$ is not a quadratic residue mod $p$.
\end{eg}

\begin{eg}
	\label{eg:blowdown}
	Let $C_1 : y_1^2 = ((x+p^2)^2 - p^{14})(x-2+p^3)$ and $C_2 : y_2^2 = ((x-p^2)^2 - p^12)(x-2-p^3)$ be hyperelliptic curves over $K$ and let $Y$ be their associated bihyperelliptic curve. The chromatic cluster picture of $Y$ consists of the \"ubereven top cluster $\Rcal$ with chromatic children, the \"ubereven cluster $\s$ with no chromatic children, and three twins $\tfrak_1, \tfrak_2$ and $\tfrak_3$, the first two black with monochromatic children and the latter chromatic with polychromatic children. Note that $\Rcal$ is principal despite being the disjoint union of two clusters as its two children are purple and black. The most subtle part of theorem is illustrated here: that the components arising from $\Rcal$, and the loops arising from $\tfrak_1$ and $\tfrak_2$ link to different pairs of $\Gamma_{\s}^{+,+}, \Gamma_{\s}^{+,-}, \Gamma_{\s}^{-,+}$ and $\Gamma_{\s}^{-,-}$; for example, $\Gamma_{\Rcal}^{+}$ links to $\Gamma_{\s}^{+,+}$ and $\Gamma_{\s}^{-,-}$, whereas $L_{\tfrak_1}^+$ links to $\Gamma_{\s}^{+,+}$ and $\Gamma_{\s}^{+,-}$. This is because $\Rcal$ is \"ubereven with chromatic children whereas $\tfrak_1$ is not \"ubereven with monochromatic (red) children.
	
	
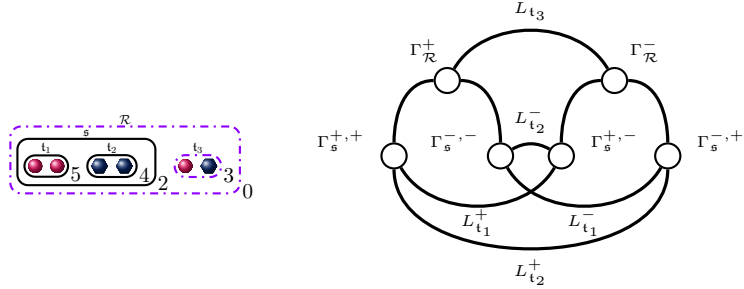
\begin{figure}[ht]
		\centering
		
		\clusterpicture
		\Root[A] {} {first} {r1};
		\Root[A] {} {r1} {r2};
		\ClusterLDName c1[][5][\mathfrak{t}_1] = (r1)(r2);
		\Root[B] {} {c1} {r3};
		\Root[B] {} {r3} {r4};
		\ClusterLDName c2[][4][\mathfrak{t}_2] = (r3)(r4);
		\ClusterLDName c3[][2][\mathfrak{s}] = (c1)(c2);
		\Root[A] {} {c3}{r5};
		\Root[B] {} {r5}{r6};
		\ClusterLDName c4[][3][\tfrak_3] = (r5)(r6), clB;
		\ClusterLDName c5[][0][\Rcal] = (c3)(c4), clB;
		\endclusterpicture \quad \quad
		\raisebox{-5em}{\begin{tikzpicture}
			\tikzstyle{every node}=[font=\tiny]
			
			\node at (1.5, 0.4) {$\Gamma_{\Rcal}^-$};
			\node at (-1.4, 0.4) {$\Gamma_{\Rcal}^+$};
			\node at (-2.5, -0.8) {$\Gamma_{\s}^{+,+}$};
			\node at (-1, -0.8) {$\Gamma_{\s}^{-,-}$};
			\node at (1.1, -0.8)  {$\Gamma_{\s}^{+,-}$};
			\node at (2.5, -0.8)  {$\Gamma_{\s}^{-,+}$};
			
			\begin{scope}[every node/.style={circle,thick,draw}]
			\node (A) at (-1.1,0) {};
			\node (A2) at (1.1,0) {};
			\node (C) at (-1.8, -1) {};
			\node (D) at (-0.4, -1) {};
			\node (E) at (0.4, -1) {};
			\node (F) at (1.8, -1) {};
			\end{scope}
			
			\path (A) edge[very thick, out=60, in=120] node[above] {$L_{\tfrak_3}$} (A2);
			\path (A) edge[very thick, out=-180, in=90] (C);
			\path (A2) edge[very thick, out=0, in=90] (F);
			\path (A2) edge[very thick, out=180, in=90] (E);
			\path (A) edge[very thick, out=0, in=90] (D);
			\path (C) edge[very thick, out=-60, in=-120] node[below, yshift=2.7] {$L^+_{\tfrak_1}$} (E);
			\path (C) edge[very thick, out=-90, in=-90] node[below] {$L^+_{\tfrak_2}$} (F);
			\path (D) edge[very thick, out=-60, in=-120] node[below, yshift=2.7] {$L^-_{\tfrak_1}$} (F);
			\path (D) edge[very thick, out=30, in=150] node[above] {$L^-_{\tfrak_2}$} (E);
			
			\end{tikzpicture}}
	\end{figure}
\end{eg}

\begin{remark}
	The above is an example of a curve whose minimal regular model has a special fibre with a non-planar dual graph (its dual graph is a $K_{3,3}$). It is in fact an example of minimal genus, since the special fibre is totally degenerate (all of the components have genus $0$). 
\end{remark}

\section{Proof}
\label{sec:proof}

The strategy of proof will be as follows: let $\X$ be the minimal model of $\PP^1$ which separates the branch points of the map $Y \rightarrow \PP^1$. We construct such a model from $\Sigma_h$ following the techniques of \cite[Sections~3-4]{DDMM18}. Normalising $\X$ in the function field $K(Y)$ gives a regular model $\Y$ of $Y$, which results in a semistable model $\Y^{\textrm{min}}$ after blowing down components of multiplicity greater than $1$. We also obtain models $\mathcal{C}_1, \mathcal{C}_h, \mathcal{C}_2$ of $C_1, C_h$ and $C_2$ respectively by normalising in the appropriate function fields. The special fibres of these intermediate models are computed using the result of \cite[Sections~5-6]{DDMM18}.

\begin{notation}
	There are many cases where we shall wish to refer to some object associated to a cluster $\s$ for each of the curves $\PP^1, C_1, C_2, C_h$ and $Y$. For example, we may wish to refer to the components arising from $\s$. In this case, the component(s) in $\Y_k$ will appear without subscript: $\Gamma_{\s}$, and those in $\X_k$ (resp. $\mathcal{C}_{1,k}, \mathcal{C}_{h,k}, \mathcal{C}_{2,k}$) will be denoted $\Gamma_{\s, \PP^1}$ (resp. $\Gamma_{\s,1}, \Gamma_{\s,h}, \Gamma_{\s,2}$).
\end{notation}

\begin{remark}
	The models $\mathcal{C}_1, \mathcal{C}_2$ and $\mathcal{C}_h$ are not necessarily \textit{minimal} models, as they come from the red (resp. blue resp. chromatic) cluster picture of $Y$ - the admissible collection of disks arising from these is \textit{not} in general the same as that arising from the cluster picture of $C_1$ (resp. $C_2$, $C_h$).
\end{remark}

\begin{lemma}
	\label{lem:ramextmodel}
	Let $C$ be a curve over $\Kur$ with semistable reduction and let $\mathcal{C}$ be a semistable model of $C$. Let $L/K$ be a totally ramified extension of degree $e$. Then $\mathcal{C} \times_{\OO_K} \OO_L$ is a semistable model of $C \times_K L$ with the same dual graph as $\mathcal{C}$, but with the lengths of all edges multiplied by $e$.
\end{lemma}
\begin{proof}
	This follows from \cite[Corollaries~10.3.36,10.3.25]{liu2006algebraic}, noting that the thickness (\cite[Definition~10.3.23]{liu2006algebraic}) of all double points multiplies by $e$ after extending the field.
\end{proof}

This allows us to happily move between ramified extensions of $K$ when describing the semistable model, which will prove invaluable.

\begin{lemma}
	Let $Y/K$ be a bihyperelliptic curve with semistable reduction. Then $\Y$, the normalisation of $\X$ in $K(Y)$, is a proper regular model of $Y$.
\end{lemma}
\begin{proof}
	The normalisation of $\X$ in $K(Y)$ is isomorphic to the normalisation of $\mathcal{C}_1$ in $K(Y)$, so it is sufficient to prove that the latter is a proper regular model of $Y$. Let $\varphi : C_1 \rightarrow \PP^1$ be the canonical double cover and write $\mathcal{D} = (\varphi_1^{-1}(f_2)) = \sum m_i \Gamma_i$, the divisor of (the pullback of) $f_2$ on $\mathcal{C}_1$. By \cite[Lemma~2.1]{srinivasan2015conductors}, it is sufficient to prove that
	\begin{enumerate}
		\item $\varphi_1^{-1}(f_2)$ is square-free on $\mathcal{C}_1$,
		\item any two $\Gamma_i$ for which $m_i$ is odd do not intersect and,
		\item any $\Gamma_i$ for which $m_i$ is odd is regular.
	\end{enumerate}
	Since $f_1$ and $f_2$ share no common roots, $\varphi_1^{-1}(f_2)$ is square-free on $\mathcal{C}_1$ and the horizontal components of $\mathcal{D}$ do not intersect. We are left to consider the vertical components of $\mathcal{D}$. Note that any vertical component of odd multiplicity must arise as the preimage of some $E \in (f_2)_{\mathrm{vert}}$ which appears with odd multiplicity. The component $E$ has either one or two preimages in $\mathcal{C}_1$. In the first case, the preimage $\Gamma$ is regular by \cite[Theorem~5.2]{DDMM18}. Since $E$ does not intersect any other component of $(f_2)$ of odd multiplicity, $\Gamma$ cannot intersect a component of $\mathcal{D}$. In the second case, the two components are still regular and do not intersect each other, and cannot intersect any other component $\mathcal{D}$ as $E$ does not intersect any other component of $(f_2)$ of odd multiplicity.
\end{proof}

\begin{prop}
	\label{prop:comps}
	Let $Y$ be a semistable bihyperelliptic curve as in Theorem \ref{thm:main} and $\Y$ the model obtained by normalising. Let $\Y^{\min}$ be the minimal regular model of $\Y$. Then each principal cluster $\s$ contributes the following components to $\Y_k^{\mathrm{min}}$: if \"ubereven, $1$ component $\Gamma_{\s}$ if $\s$ has polychromatic children and $2$ components $\Gamma_{\s}^+, \Gamma_{\s}^-$ if $\s$ has monochromatic children; and if \"ubereven: $2$ components $\Gamma_{\s}^+, \Gamma_{\s}^-$ if $\s$ has chromatic children and $4$ components $\Gamma_{\s}^{+,+}, \Gamma_{\s}^{+,-}, \Gamma_{\s}^{-,+}, \Gamma_{\s}^{-,-}$ if $\s$ has no chromatic children.
\end{prop}
\begin{proof}
	Consider a principal cluster $\s$ and its corresponding component $\Gamma_{\s, \PP^1}$ in the model of $\PP^1$. If $\s$ is not \"ubereven, then $\Gamma_{\s, \PP^1}$ lifts to one component $\Gamma_{\s,h}$ in $\mathcal{C}_h$ and so lifts to either one or two components in $\Y_k$. Suppose $\s$ has polychromatic children. Then $\s$ has one corresponding component $\Gamma_{\s,1} \in \mathcal{C}_1$, and $\Gamma_{\s,1}$ contains branch points of the morphism $\Y \rightarrow \mathcal{C}_1$ (corresponding to the blue children of $\s$), and hence must lift to a single component in $\Y_k$.  If $\s$ has monochromatic (e.g. red) children then it has two components in either $\mathcal{C}_1$ or $\mathcal{C}_2$ (in this case $\mathcal{C}_2$), and hence must have two associated components in $\Y$. 
	
	If $\s$ is \"ubereven, then it has two components in $\mathcal{C}_h$ so lifts to either two or four in $\Y_k$. If $\s$ has chromatic (e.g. red) children then it has a single component in either $\mathcal{C}_1$ or $\mathcal{C}_2$ (in this case $\mathcal{C}_1$), so can only lift to two components in $\Y_k$. If $\s$ has no chromatic children, then it has two corresponding components in each of $\mathcal{C}_1, \mathcal{C}_2$ and $\mathcal{C}_h$. Since $C_2 \times C_2$ acts on the components of $\Y_k$ arising from $\s$, and their images are the $\Gamma_{\s,i}^{\pm}$ under the quotient of the three non trivial subgroups of $C_2 \times C_2$, there must be four components corresponding to $\s$ in $\Y_k$.
	
	Since $\s$ is principal, it is principal in one of $\Sigma_1, \Sigma_2$ or $\Sigma_h$, say $\Sigma_1$. Then by \cite[Theorem~8.5]{DDMM18}, any component $\Gamma_{\s,1} \in \mathcal{C}_1$ arising from $\s$ either has positive genus, or intersects at least three other components. Therefore the same can be said for any $\Gamma \in \Y_k$ arising from $\s$. We cannot blow such components down, and hence the same components appear in $\Y^{\min}_k$.
\end{proof}

\begin{prop}
	\label{prop:links}
	Let $Y,\Y^{\min}$ be as in Proposition \ref{prop:comps}, and let $\s' < \s$ be principal clusters of $Y$. Then the components of $\s$ and $\s'$ in $\Y_k$ are linked by two chains if $\s'$ is chromatic and four otherwise, as described in the statement of Theorem \ref{thm:main}. Furthermore, if $\tfrak < \s$ is a twin or $\s < \tfrak$ a cotwin then there is one loop if the child is chromatic and two loops if the child is black, and if $\Rcal = \s \sqcup \s'$ is not principal then the components of $\s$ and $\s'$ are linked as in the statement of Theorem \ref{thm:main}.
\end{prop}
\begin{proof}
	Assume that we are in the case where $\s' < \s$ are both principal, since the other cases are checked similarly using Remark \ref{rmk:twincomps}. It is clear that components have linking chains to the components corresponding to their parents, since this is the case for the model of $\PP^1$. Therefore, we have to calculate how many linking chains there are ($2$ or $4$) and precisely which components are linked to which others. The lengths of the linking chains is proved separately in Lemma \ref{lem:linkinglengths}.
	
	Suppose that $\s' < \s$ with $\s'$ chromatic. Then there are several cases for the different children $\s'$ can have. If $\s'$ is red, then $\Gamma_{\s,1}$ and $\Gamma_{\s',1}$ are linked by one chain and $\Gamma_{\s,2}^{\pm}$ and $\Gamma_{\s',2}^{\pm}$ are linked by two chains. Similarly if $\s'$ is blue (swapping $1$ and $2$). In either case, $\Gamma_{\s}^{\pm}$ and $\Gamma_{\s'}^{\pm}$ are linked by two chains. Similarly if $\s'$ is purple then $\Gamma_{\s,h}$ and $\Gamma_{\s',h}$ are linked by two chains but $\Gamma_{\s,1}$ and $\Gamma_{\s',1}$ are linked by one chain so we get two linking chains upstairs.
	
	Now suppose that $\s'$ is black. Then $\Gamma_{\s',1}^{\pm}$ has two linking chains up to the components of its parent, as does $\Gamma_{\s',2}$, so by a similar argument to Proposition \ref{prop:comps} the components of $\s'$ have four linking chains up to the components of $\s$. Therefore the number of linking chains is correct. We must check that the correct components are linked.
	
	This is done on a case by case basis. First assume $\s, \s'$ are not \"ubereven. If $\s$ has polychromatic children then it only has one component and everything is correct up to relabelling. Similarly if $\s'$ has polychromatic children. So assume both $\s$ and $\s'$ have monochromatic children. If they have monochromatic children of the same colour (say red), their components are linked in the following way in the tower of models:
	
	
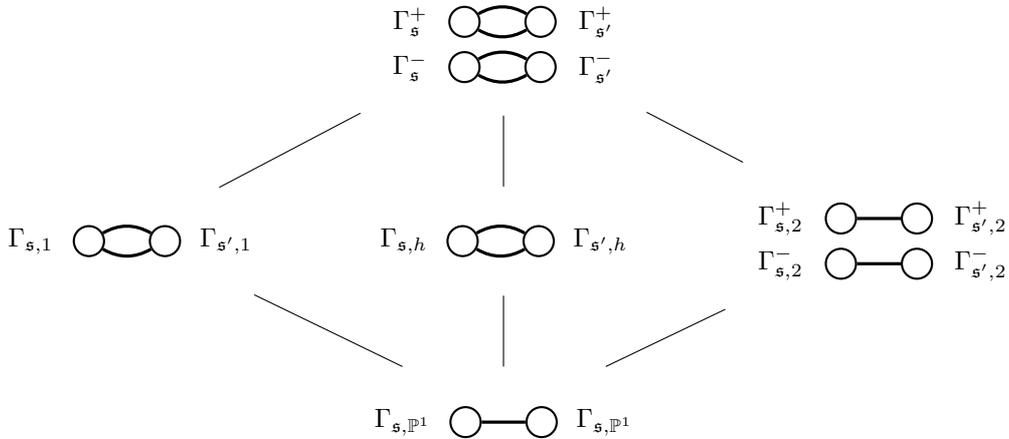
\begin{figure}[h]
		\centering
		\begin{tikzcd}
			& {\begin{tikzpicture}
				
				\begin{scope}[every node/.style={circle,thick,draw}]
				\node[label=left:{\Gamma_{\s}^+}] (A) at (0,0.3) {};
				\node[label=right:{\Gamma_{\s'}^+}] (B) at (1,0.3) {}; 
				\node[label=left:{\Gamma_{\s}^-}] (C) at (0, -0.3) {};
				\node[label=right:{\Gamma_{\s'}^-}] (D) at (1, -0.3) {};
				\end{scope}
				
				\path (A) edge[very thick, out=30, in=150] (B);
				\path (A) edge[very thick, out=-30, in=-150] (B);
				\path (C) edge[very thick, out=30, in=150] (D);
				\path (C) edge[very thick, out=-30, in=-150] (D);
				
				\end{tikzpicture}} \arrow[ld, no head] \arrow[d, no head] \arrow[rd, no head] & \\ {\begin{tikzpicture}
				\begin{scope}[every node/.style={circle,thick,draw}]
				\node[label=left:{\Gamma_{\s,1}}] (A) at (0,0) {};
				\node[label=right:{\Gamma_{\s',1}}] (B) at (1,0) {};
				\end{scope}
				
				\path (A) edge[very thick, out=30, in=150] (B);
				\path (A) edge[very thick, out=-30, in=-150] (B);
				
				\end{tikzpicture}} \arrow[rd, no head] & {\begin{tikzpicture}
				\begin{scope}[every node/.style={circle,thick,draw}]
				\node[label=left:{\Gamma_{\s,h}}] (A) at (0,0) {};
				\node[label=right:{\Gamma_{\s',h}}] (B) at (1,0) {};
				\end{scope}
				
				\path (A) edge[very thick, out=30, in=150] (B);
				\path (A) edge[very thick, out=-30, in=-150] (B);  
				
				\end{tikzpicture}} \arrow[d, no head] & {\begin{tikzpicture}
				\begin{scope}[every node/.style={circle,thick,draw}]
				\node[label=left:{\Gamma_{\s,2}^+}] (A) at (0,0.3) {};
				\node[label=right:{\Gamma_{\s',2}^+}] (B) at (1,0.3) {}; 
				\node[label=left:{\Gamma_{\s,2}^-}] (C) at (0, -0.3) {};
				\node[label=right:{\Gamma_{\s',2}^-}] (D) at (1, -0.3) {};
				\end{scope}
				
				\path (A) edge[very thick, out=0, in=180] (B);
				\path (C) edge[very thick, out=0, in=180] (D);
				
				\end{tikzpicture}} \arrow[ld, no head]\\
			& {\begin{tikzpicture}
				\begin{scope}[every node/.style={circle,thick,draw}]
				\node[label=left:{\Gamma_{\s, \PP^1}}] (A) at (0,0) {};
				\node[label=right:{\Gamma_{\s, \PP^1}}] (B) at (1, 0) {};
				\end{scope}
				
				\path (A) edge[very thick, out=0, in=180] (B);
				
				\end{tikzpicture}}
		\end{tikzcd}
	\end{figure}

	If $\s$ and $\s'$ have monochromatic children of different colours (e.g. $\s$ red and $\s'$ blue), then their components are linked in the following way in the tower of models:

	
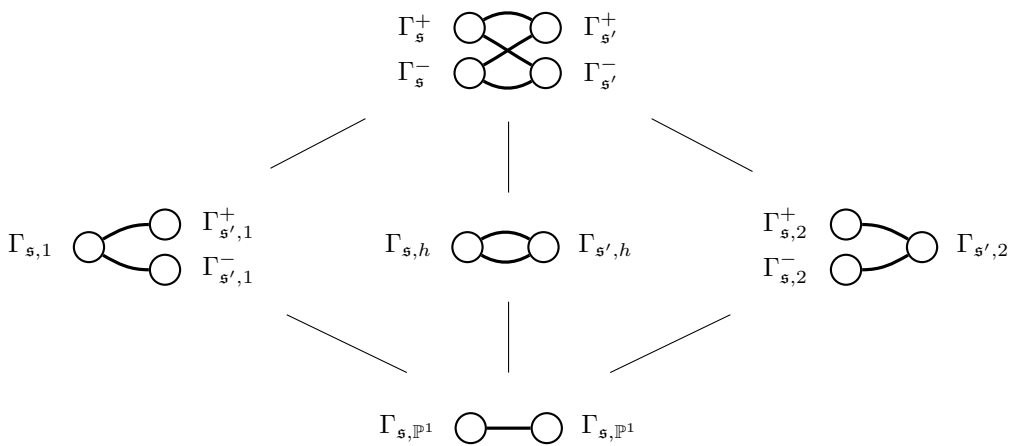
\begin{figure}[h]
		\centering
		\begin{tikzcd}
			& {\begin{tikzpicture}
				
				\begin{scope}[every node/.style={circle,thick,draw}]
				\node[label=left:{\Gamma_{\s}^+}] (A) at (0,0.3) {};
				\node[label=right:{\Gamma_{\s'}^+}] (B) at (1,0.3) {}; 
				\node[label=left:{\Gamma_{\s}^-}] (C) at (0, -0.3) {};
				\node[label=right:{\Gamma_{\s'}^-}] (D) at (1, -0.3) {};
				\end{scope}
				
				\path (A) edge[very thick, out=30, in=150] (B);
				\path (A) edge[very thick, out=-30, in=-210] (D);
				\path (C) edge[very thick, out=30, in=210] (B);
				\path (C) edge[very thick, out=-30, in=-150] (D);
				
				\end{tikzpicture}} \arrow[ld, no head] \arrow[d, no head] \arrow[rd, no head] & \\ {\begin{tikzpicture}
				\begin{scope}[every node/.style={circle,thick,draw}]
				\node[label=left:{\Gamma_{\s,1}}] (A) at (0,0) {};
				\node[label=right:{\Gamma_{\s',1}^+}] (B) at (1,0.3) {};
				\node[label=right:{\Gamma_{\s',1}^-}] (C) at (1, -0.3) {};
				\end{scope}
				
				\path (A) edge[very thick, out=30, in=180] (B);
				\path (A) edge[very thick, out=-30, in=-180] (C);
				
				\end{tikzpicture}} \arrow[rd, no head] & {\begin{tikzpicture}
				\begin{scope}[every node/.style={circle,thick,draw}]
				\node[label=left:{\Gamma_{\s,h}}] (A) at (0,0) {};
				\node[label=right:{\Gamma_{\s',h}}] (B) at (1,0) {};
				\end{scope}
				
				\path (A) edge[very thick, out=30, in=150] (B);
				\path (A) edge[very thick, out=-30, in=-150] (B);  
				
				\end{tikzpicture}} \arrow[d, no head] & {\begin{tikzpicture}
				\begin{scope}[every node/.style={circle,thick,draw}]
				\node[label=left:{\Gamma_{\s,2}^+}] (A) at (0,0.3) {};
				\node[label=right:{\Gamma_{\s',2}}] (B) at (1,0) {}; 
				\node[label=left:{\Gamma_{\s,2}^-}] (C) at (0, -0.3) {};

				\end{scope}
				
				\path (A) edge[very thick, out=0, in=150] (B);
				\path (C) edge[very thick, out=0, in=210] (B);
				
				\end{tikzpicture}} \arrow[ld, no head]\\
			& {\begin{tikzpicture}
				\begin{scope}[every node/.style={circle,thick,draw}]
				\node[label=left:{\Gamma_{\s, \PP^1}}] (A) at (0,0) {};
				\node[label=right:{\Gamma_{\s, \PP^1}}] (B) at (1, 0) {};
				\end{scope}
				
				\path (A) edge[very thick, out=0, in=180] (B);
				
				\end{tikzpicture}}
		\end{tikzcd}
	\end{figure}
	
	From here the following notation arises: that $\Gamma_{\s}^{\pm,+} = \Gamma_{\s}^{\pm,-} = \Gamma_{\s}^{\pm}$ if $\s$ has monochromatic red children and $\Gamma_{\s}^{+,\pm} = \Gamma_{\s}^{-,\pm} = \Gamma_{\s}^{\pm}$ if $\s$ has monochromatic blue children. If $\s$ or $\s'$ is \"ubereven, the different cases can be checked similarly.
	
\end{proof}

\begin{lemma}
	\label{lem:linkinglengths}
	Let $Y,\Y^{\min}$ be as in Proposition \ref{prop:comps}, and let $\s' < \s$ be two principal clusters. Then any linking chain $L_{\s'}$ arising from this pair has length $\frac12 \delta_{\s'}$ if $\s'$ is chromatic and $\delta_{\s'}$ otherwise. If $\tfrak < \s$ is a twin or $\s < \tfrak$ a cotwin then any loop arising from $\tfrak$ has length $\delta_{\tfrak}$ if the child is chromatic, and $2\delta_{\tfrak}$ otherwise. If $\Rcal = \s \sqcup \s'$ is not principal, then any length of a linking chain arising from $\s$ and $\s'$ has length as described in Theorem \ref{thm:main}.
\end{lemma}
\begin{proof}
	Possibly after a finite extension $L/\Kur$, the map $Y \rightarrow \PP^1$ extends to a map of models $\Y \rightarrow \X$ by \cite[Theorem~2.3]{liu_lorenzini_1999}. Furthermore, this map induces a harmonic morphism of augmented $\Z$-graphs, in the sense of \cite[Section~2]{amini2015lifting} (see also Sections 5,8) on the dual graphs of $\Y$ and $\X$. The length of an edge between two vertices in an augmented $\Z$-graph is the thickness of the intersection point of the components the vertices represent. Therefore the distance between two vertices of degree $\geq3$ is exactly the length of the linking chain between them. 
	
	If $\s' < \s$ are principal, the lemma follows, noting that a linking chain from a chromatic cluster to its parent has two preimages in $\Y'$ (so the length halves), but a linking chain from a black cluster to its parent has four preimages so the length stays the same.
	
	If $\tfrak < \s$ is a chromatic twin then (possibly after a field extension), we can think of the loop $L_{\tfrak}$ as consisting of a component $\Gamma_{\tfrak}$, the unique lift of $\Gamma_{\tfrak, \PP^1}$ with two linking chains to $\Gamma_{\s}$. Since $\tfrak$ is chromatic, by the argument above the linking chains will both have length $\frac12 \delta_{\tfrak}$ and hence the total loop will have length $\delta_{\tfrak}$. A similar argument is made if $\tfrak$ is a black twin, or if $\Rcal = \s \sqcup \s'$ is not principal.
%
%
\end{proof}

\begin{prop}
	Let $Y,\Y^{\min}$ be as in Proposition \ref{prop:comps} and let $\s$ be a principal cluster of $Y$. Then the components associated to $\s$ have genus $g_{\chi}(\s)$.
\end{prop}
\begin{proof}
	Let $\s$ be a principal cluster. First suppose $\s$ is not \"ubereven and has polychromatic children. In this case there is a unique component $\Gamma_{\s}$ arising from $\s$. This is then a direct application of Riemann-Hurwitz. The children of $\s$ correspond to points on the component $\Gamma_{\s, \PP^1}$ as in \cite[Definition~3.7]{DDMM18}, and by Proposition \ref{prop:links} the points arising from chromatic children are precisely the non-infinity branch points of $\Gamma_{\s} \rightarrow \Gamma_{\s,\PP^1}$. In addition, there is an extra branch point at infinity, unless $\s$ is the top cluster and $f_1$ and $f_2$ both have even degree, or the top cluster $\Rcal = \s \bigsqcup \s'$ is not principal, $\s$ is even and $f_1$ and $f_2$ both have even degree. 
	
	If $\s$ is not \"ubereven and has monochromatic children, then the components $\Gamma_{\s}^+$ and $\Gamma_{\s}^-$ are each isomorphic to $\Gamma_{\s,h}$ and so have genus $g(\s)$. The same is true if $\s$ is \"ubereven and has chromatic children, except with $\Gamma_{\s,1}$ if $\s$ has red children and $\Gamma_{\s,2}$ if blue. If $\s$ is \"ubereven with no chromatic children then its $4$ components must have genus $0=g(\s)$ as well.
\end{proof}

\begin{thm}
	Denote the Frobenius automorphism by $\Frobtwo$. It acts on the dual graph of $\Y^{\min}_k$ in the following way:
	
	\begin{enumerate}
		\item $\Frobtwo(\Gamma_{\s}^{\pm}) = \Gamma_{\Frobtwo(\s)}^{\pm\epsilon_{\s,i}(\Frobtwo)}$ \textrm{for $\s$ with chromatic children, $i \in \{1,2,h\}$ with $\s \in \Sigma_i$ \"ubereven},
		\item $\Frobtwo(\Gamma_{\s}^{\pm,\pm}) = \Gamma_{\Frobtwo(\s)}^{\pm\epsilon_{\s,2}(\Frobtwo),\pm\epsilon_{\s,1}(\Frobtwo)}$ \textrm{for $\s$ \"ubereven with no chromatic children},
		\item $\Frobtwo(L_{\s}^{\pm}) = \Gamma_{\Frobtwo(\s)}^{\pm\epsilon_{\s,i}(\Frobtwo)}$ \textrm{for $\s$ chromatic, $i \in \{1,2,h\}$ with with $\s \in \Sigma_i$ even},
		\item $\Frobtwo(L_{\s}^{\pm,\pm}) = L_{\Frobtwo(\s)}^{\pm\epsilon_{\s,2}(\Frobtwo),\pm\epsilon_{\s,1}(\Frobtwo)}$ \textrm{for $\s$ black},
		\item $\Frobtwo(L_{\tfrak}) = \epsilon_{\tfrak, h}(\Frobtwo) L_{\Frobtwo(\tfrak)}$ for $\tfrak$ a chromatic twin, where $-L$ denotes $L$ with the opposite orientation,
		\item $\Frobtwo(L_{\tfrak}^{\pm}) = \epsilon_{\tfrak, j}(\Frobtwo) L_{\Frobtwo(\tfrak)}^{\pm \epsilon_{\tfrak, i}(\Frobtwo)}$ for $\tfrak$ a black twin, $i,j \in \{1, 2\}$ such that $\tfrak$ is empty in $\Sigma_i$ and $i \neq j$.
	\end{enumerate}
\end{thm}
\begin{proof}
	By Proposition \ref{prop:comps}, the components we must blow down to obtain $\Y^{\min}_k$ from $\Y_k$ are all in linking chains, so it is sufficient to calculate the action of Frobenius on $\Y_k$ as the action on the shortened linking chains is the same as the originals. The action of Frobenius commutes with the quotient maps, so we can deduce the action of Frobenius on the components of $\Y_k$ from the corresponding action of Frobenius on $\mathcal{C}_1, \mathcal{C}_2$ and $\mathcal{C}_h$, which is known by \cite[Theorem~8.5]{DDMM18}. First we focus on clusters. For a principal cluster $\s$, the set $ \mathcal{E}_{\s} = \{\Gamma_{\s}^{\pm, \pm}\}$ is mapped to $\mathcal{E}_{\phi(\s)} = \{\Gamma_{\phi(\s)}^{\pm, \pm} \}$ by Frobenius, since Frobenius maps the images of $\mathcal{E}_{\s}$ to the images of $\mathcal{E}_{\phi(\s)}$ in  $\mathcal{C}_1, \mathcal{C}_2$ and $\mathcal{C}_h$. It remains to show which component of $\mathcal{E}_{\s}$ is mapped to which of $\mathcal{E}_{\phi(\s)}$. 
	
	If $\s$ is a principal cluster with polychromatic children then $\mathcal{E}_{\s}$ consists of one component and hence there is nothing to verify. If $\s$ is a principal cluster with monochromatic, red children (so therefore $\s \in \Sigma_2$ is even) then there are two components $\Gamma_{\s}^+, \Gamma_{\s}^-$ corresponding to $\s$, which are the lifts of two components $\Gamma_{\s,2}^+, \Gamma_{\s,2}^- \in \mathcal{C}_{2,k}$. Therefore $\phi$ acts on $\Gamma_{\s}^{\pm}$ as it does on $\Gamma_{\s, 2}^{\pm}$. But by \cite[Theorem~8.5]{DDMM18}, $\phi(\Gamma_{\s, 2}^{\pm}) = \Gamma_{\phi(\s), 2}^{\pm \epsilon_2(\phi)}$, and so $\phi(\Gamma_{\s}^{\pm}) = \Gamma_{\phi(\s)}^{\pm \epsilon_2(\phi)}$. Similarly if $\s$ has monochromatic, blue children or is \"ubereven with polychromatic children.
	
	Now suppose that $\s$ is an \"ubereven cluster with no chromatic children. In this case there are two components arising from $\s$ in $\mathcal{C}_{i,k}$ for $i=1,2$, $\Gamma_{\s,i}^+$ and $\Gamma_{\s,i}^-$. Consider $i=1$. In $\mathcal{C}_1$, the action of Frobenius is $\phi(\Gamma_{\s, 1}^{\pm}) = \Gamma_{\phi(\s), 1}^{\pm\epsilon_1(\phi)}$. The component $\Gamma_{\s,1}^+$ lifts to the components $\Gamma_{\s}^{+, +}$ and $\Gamma_{\s}^{-, +}$ and so the set $\{\Gamma_{\s}^{+,\pm}, \Gamma_{\s}^{-,\pm}\}$ is mapped to $\{\Gamma_{\s}^{+, \pm\epsilon_1{\phi}}, \Gamma_{\s}^{-, \pm\epsilon_1{\phi}}\}$. The same argument for $i=2$ implies that the set $\{\Gamma_{\s}^{\pm, +}, \Gamma_{\s}^{\pm, -} \} $ is mapped to the set $\{\Gamma_{\s}^{\pm \epsilon_2(\phi), +}, \Gamma_{\s}^{\pm \epsilon_2(\phi), -} \}$. Combining these gives the action of Frobenius. 
	
	For linking chains between components of principal clusters $\s' < \s$, the action on the whole linking chain is determined by the action on any component in the linking chain. Suppose $D$ is a $p$-adic disk with $\s' < D < \s$. If $\s'$ is chromatic (i.e., case (iii)) then $\Gamma_{D,\PP^1}$ has two preimages in $\Y$ and these are permuted like the principal components in (i). If $\s'$ is black (case (iv)), then $\Gamma_{D,\PP^1}$ has four preimages in $\Y_k$ and these are permuted like the principal components in case (ii).
	
	Loops associated to twins and linking chains between $\s$ and $\s'$ when $\Rcal = \s \sqcup \s'$ are not principal can be dealt with using Remark \ref{rmk:twincomps}, as in proofs in the rest of the section.
\end{proof}


\bibliographystyle{plain}
\bibliography{bibli}

\end{document}